\documentclass{amsart}
\usepackage{amsfonts}
\usepackage{graphicx}
\usepackage{amscd}
\usepackage{amsmath}
\usepackage{amssymb}

\makeatletter
\@namedef{subjclassname@2010}{%
  \textup{2010} Mathematics Subject Classification}
\makeatother

\theoremstyle{plain}
\newtheorem*{acknowledgement}{Acknowledgement}

\newtheorem{lemma}{\bf Lemma}

\newtheorem{remark}{Remark}

\newtheorem{theorem}{\bf Theorem}
\newcommand{\ric}{\mathring{Ric}}

\newcommand{\rc}{\mathring{R}}
\renewcommand{\div}{div}

\numberwithin{equation}{section}

\begin{document}
\title[Critical metrics of the volume functional]{Volume functional of compact $4$-manifolds with a prescribed boundary metric}

\author{H. Baltazar}
\author{R. Di\'ogenes}
 \author{E. Ribeiro Jr}

\address[H. Baltazar]{Universidade Federal do Piau\'{\i} - UFPI, Departamento de Matem\'{a}tica, 64049-550, Te\-re\-si\-na - PI, Brazil.}
\email{halyson@ufpi.edu.br}

\address[R. Di\'ogenes]{UNILAB, Instituto de Ci\^encias Exatas e da Natureza, 62785-000, Acarape - CE, Brazil.}
\email{rafaeldiogenes@unilab.edu.br}

\address[E. Ribeiro Jr]{Universidade Federal do Cear\'a - UFC, Departamento  de Matem\'atica, Campus do Pici, Av. Humberto Monte, Bloco 914,
60455-760, Fortaleza - CE, Brazil.}
\email{ernani@mat.ufc.br}

\thanks{H. Baltazar was partially supported by CNPq/Brazil and FAPEPI/Brazil}

\thanks{E. Ribeiro was partially supported by CNPq/Brazil [Grant: 305410/2018-0 and 160002/2019-2], PRONEX - FUNCAP /CNPq/ Brazil and CAPES/ Brazil - Finance Code 001}

\subjclass[2010]{Primary 53C25, 53C20, 53C21; Secondary 53C65}
\keywords{volume functional; critical metrics; compact manifolds; boundary}
\date{April 4, 2019}

\begin{abstract}
We prove that a critical metric of the volume functional on a $4$-dimensional compact manifold with boundary satisfying a second-order vani\-shing condition on the Weyl tensor must be isometric to a geodesic ball in a simply connected space form $\mathbb{R}^{4}$, $\mathbb{H}^{4}$ or $\mathbb{S}^{4}.$ Moreover, we provide an integral curvature estimate involving the Yamabe constant for critical metrics of the volume functional, which allows us to get a rigidity result for such critical metrics. 
\end{abstract}

\maketitle

\section{Introduction}
\label{intro}

Let $(M^{n},\,g)$ be a connected compact Riemannian manifold with dimension $n$ at least three. Following the terminology adopted by Corvino, Eichmair and Miao \cite{CEM} as well as Miao and Tam \cite{miaotam}, we say that $g$ is a {\it $V$-static metric} if there is a smooth function $f$ on $M^n$ and a constant $\kappa$ satisfying the $V$-static equation
\begin{equation}
\label{eqVstatic} \mathfrak{L}_{g}^{*}(f)=-(\Delta f)g+Hess\, f-fRic=\kappa g,
\end{equation} where $\mathfrak{L}_{g}^{*}$ is the formal $L^{2}$-adjoint of the linearization of the scalar curvature ope\-rator $\mathfrak{L}_{g},$ which plays a fundamental role in problems related to prescribing the scalar curvature function. Here, $Ric,$ $\Delta$ and $Hess$ stand, respectively, for the Ricci tensor, the Laplacian operator and the Hessian form on $(M^n,\,g).$ Such a function $f$ is called a {\it $V$-static potential}. 

In the work \cite{CEM}, Corvino, Eichmair and Miao proved that $V$-static metrics arises from the modified problem of finding stationary points for the volume functional on the space of metrics whose scalar curvature is equal to a given constant and such metrics are useful as an attempt to better understand the interplay between scalar curvature and volume. In this context, Corvino, Eichmair and Miao \cite{CEM} were able to show that when the metric $g$ does not admit non-trivial solution to Eq. (\ref{eqVstatic}), then one can achieve simultaneously a prescribed perturbation of the scalar curvature that is compactly supported in a bounded domain $\Omega$ and a prescribed perturbation of the volume by a small deformation of the metric in $\overline{\Omega}.$ It is worth to point out that a Riemannian manifold $(M^{n},\,g)$ for which there exists a nontrivial function $f$ satisfying (\ref{eqVstatic}) must have constant scalar curvature $R$ (cf. Proposition 2.1 in \cite{CEM} and Theorem 7 in \cite{miaotam}).

The case where $\kappa\neq 0$ in (\ref{eqVstatic}) and the zero-set of $f$ is the boundary $\partial M$ was studied by Miao and Tam \cite{miaotam}. To be precise, Miao and Tam \cite{miaotam} showed that these critical metrics arise as critical points of the volume functional on $M^n$ when restricted to the class of metrics $g$ with prescribed constant scalar curvature such that $g_{|_{T \partial M}}=h$ for a prescribed Riemannian metric $h$  on the boundary. In this setting, a {\it Miao-Tam critical metric} is a 3-tuple $(M^n,\,g,\,f),$ where $(M^{n},\,g)$ is a compact Riemannian manifold of dimension at least three with a smooth boundary $\partial M$ and $f: M^{n}\to \Bbb{R}$ is a smooth function such that $f^{-1}(0)=\partial M$ satisfying the overdetermined-elliptic system
\begin{equation}
\label{eqMiaoTam} \mathfrak{L}_{g}^{*}(f)=-(\Delta f)g+Hess\, f-fRic=g.
\end{equation} For more details on such a subject, we refer the reader to \cite{BalRi,BLF,BDR,BDRR,CEM,miaotam,miaotamTAMS,yuan}.

Some explicit examples of Miao-Tam critical metrics were built on connected domain with compact closure in $\Bbb{R}^n,$ $\Bbb{H}^n$ and $\Bbb{S}^n$ (see \cite{miaotam}). Moreover, some results obtained in \cite{miaotam} suggest that critical metrics with a prescribed boundary metric seem to be rather rigid. From this perspective, it is natural to ask whether these quoted examples are the only Miao-Tam critical metrics.

In order to motivate our main results, we now briefly recall a few relevant partial answers to this question under vanishing conditions involving zero, first, or specific second order derivatives of the Weyl tensor $W,$ which is defined by the following decomposition formula
\begin{eqnarray}
R_{ijkl}&=&W_{ijkl}+\frac{1}{n-2}\big(R_{ik}g_{jl}+R_{jl}g_{ik}-R_{il}g_{jk}-R_{jk}g_{il}\big) \nonumber\\
 &&-\frac{R}{(n-1)(n-2)}\big(g_{jl}g_{ik}-g_{il}g_{jk}\big),
\end{eqnarray} where $R_{ijkl}$ stands for the Riemann curvature operator. Thereby, the Ricci curvature is given by $R_{kl}=g^{ij}R_{ikjl}$ and the scalar curvature is $R=g^{kl}R_{kl}.$ Throu\-ghout the paper the Einstein convention of summing over the repeated indices will be adopted. Inspired by ideas developed by Kobayashi \cite{kobayashi}, Kobayashi and Obata \cite{obata}, Miao and Tam \cite{miaotamTAMS} proved that a locally conformally flat  (i.e. $W=0$ for $n\ge 4$ and $C=0$ in dimension $3,$ where $C$ is the Cotton tensor) simply connected, compact Miao-Tam critical metric $(M^{n},\,g,\,f)$ with boundary isometric to a standard sphere $\mathbb{S}^{n-1}$ must be isometric to a geodesic ball in a simply connected space form $\mathbb{R}^{n}$, $\mathbb{H}^{n}$ or $\mathbb{S}^{n}.$ Afterward, motivated by \cite{CaoChen}, Barros, Di\'{o}genes and Ribeiro \cite{BDR} proved that a Bach-flat (i.e. $B_{ij}=0,$ where $B_{ij}$ is given by (\ref{bach})) simply connected, $4$-dimensional compact Miao-Tam critical metric with boundary isometric to a standard sphere $\mathbb{S}^{3}$ must be isometric to a geodesic ball in a simply connected space form $\mathbb{R}^{4}$, $\mathbb{H}^{4}$ or $\mathbb{S}^{4}.$ The Bach-flat condition can be seen as a vanishing condition involving second and zero order terms in the Weyl tensor. Recently, Kim and Shin \cite{Kim} showed that  a simply connected, compact Miao-Tam critical metric with harmonic curvature and boundary isometric to a standard sphere $\mathbb{S}^{3}$ must be isometric to a geodesic ball in a simply connected space form $\mathbb{R}^{4}$, $\mathbb{H}^{4}$ or $\mathbb{S}^{4}.$ But, since a Miao-Tam critical metric has constant scalar curvature, we conclude that the assumption of harmonic curvature in the Kim-Shin result can be replaced by the harmonic Weyl tensor condition (i.e. $div(W)=\nabla_{l}W_{ikjl}=0).$ See also \cite{BalRi} and \cite{miaotamTAMS} for further related results.

Before presenting the main results, let us remark that recently Catino, Mastrolia and Monticelli \cite{catino} obtained an important classification for gradient Ricci solitons admitting a fourth-order vanishing condition on the Weyl tensor. To be precise, they showed that any $n$-dimensional $(n \ge 4)$ gradient shrinking Ricci soliton with fourth-order divergence-free Weyl tensor (i.e. $div^4(W)=\nabla_{j}\nabla_{k}\nabla_{l}\nabla_{i}W_{ijkl}=0$) is either Einstein or a finite quotient of $N^{n-k} \times \Bbb{R}^k,$ $(k > 0),$ the product of an Einstein manifold $N^{n-k}$ with the Gaussian shrinking soliton $\Bbb{R}^{k}.$

At the same time, it is well-known that dimension four displays fascinating and peculiar features, for this reason very much attention has been given to this dimension; see, for instance \cite{besse,GS,scorpan}, for more information about this specific dimension. In this paper, motivated by
 \cite{catino}, we shall classify  four-dimensional compact Miao-Tam critical metrics under the {\it second order divergence-free Weyl tensor condition} 
\begin{equation}
div^{2}(W)=\nabla_{l}\nabla_{i}W_{ijkl}=0,
\end{equation} which is clearly weaker than the locally conformally flat and harmonic Weyl tensor conditions considered in \cite{miaotamTAMS,Kim}. More precisely, we have established the following result.

\begin{theorem}\label{thmA}
Let $(M^{4},\,g,\,f)$ be a simply connected, compact Miao-Tam critical metric with second order divergence-free Weyl tensor and boundary isometric to a standard sphere $\mathbb{S}^{3}.$ Then $(M^{4},\,g)$ is isometric to a geodesic ball in a simply connected space form $\mathbb{R}^{4}$, $\mathbb{H}^{4}$ or $\mathbb{S}^{4}.$
\end{theorem}

For our second main result, we now recall the Yamabe constant for Riemannian manifolds with boundary

\begin{equation}
\label{Yamabeconst}
\mathcal{Y}(M,\partial M,[g])=\inf_{\phi\not\equiv0}\frac{\int_M\left(\frac{4(n-1)}{n-2}|\nabla\phi|^2+R\phi^2\right)dV_g+2\int_{\partial M}H\phi^2dS_g}{\left(\int_M|\phi|^\frac{2n}{n-2}dV_g\right)^\frac{n-2}{n}},
\end{equation} where $H$ is the mean curvature of $\partial M$ and $\phi$ is a smooth positive function on $M^n.$ For more details, we refer the reader to \cite{Escobar} and references therein.

In dimension 4, it is known that the Yamabe invariant alone is too weak to classify a given manifold, and for this reason one requires additional conditions to obtain a classification theorem. In \cite{catinoAdv}, Catino proved that any four-dimensional compact gradient shrinking Ricci soliton satisfying some $L^{2}$-pinching condition is isometric to a quotient of the round sphere; see also \cite{Huang} for further related results. Here, by adapting the method outlined in \cite{catinoAdv} as well as \cite{HV}, we provide an integral curvature estimate involving the Yamabe constant for four-dimensional compact Miao-Tam critical metrics. More precisely, we have the following result.

\begin{theorem}
\label{thmB} Let $(M^{4},\,g,\,f)$ be a simply connected, compact Miao-Tam critical metric with positive scalar curvature. Then we have:

\begin{eqnarray}
\label{thmeq}
\mathcal{Y}(M,\partial M,[g])\Phi(M)&\le& 4\sqrt{3}\left(\int_M\left(|W|^2+|\ric|^2\right)dV_g\right)^\frac{1}{2}\Phi(M)\nonumber\\&&+ 3(3\sqrt{2}-1)\int_M|\ric|^2|\nabla f|^2dV_g,
\end{eqnarray} where $\mathcal{Y}(M,\partial M,[g])$ is given by (\ref{Yamabeconst}) and $\Phi(M)=\left(\int_Mf^4|\ric|^4dV_g\right)^\frac{1}{2}.$ Moreover, if the equality occurs in (\ref{thmeq}), then $M^4$ is isometric to a geodesic ball in $\mathbb{S}^{4}.$
\end{theorem}

The article is organized as follows. In Section \ref{secThm1}, we review some classical tensors and basic facts. Moreover, we present the proof of Theorem \ref{thmA}. In Section \ref{secThmB}, we prove a couple of key lemmas and we present the proof of Theorem \ref{thmB}.

\section{Rigidity Result}
\label{secThm1}

In this section we will present the proof of Theorem \ref{thmA}. Before to do so, we need recall some special tensors in the study of curvature for a Riemannian manifold $(M^n,\,g),\,n\ge 3.$  The first one is the Weyl tensor $W$ which is defined by the following decomposition formula
\begin{eqnarray}
\label{weyl}
R_{ijkl}&=&W_{ijkl}+\frac{1}{n-2}\big(R_{ik}g_{jl}+R_{jl}g_{ik}-R_{il}g_{jk}-R_{jk}g_{il}\big) \nonumber\\
 &&-\frac{R}{(n-1)(n-2)}\big(g_{jl}g_{ik}-g_{il}g_{jk}\big),
\end{eqnarray} where $R_{ijkl}$ stands for the Riemann curvature operator. The second one is the Cotton tensor $C$ given by
\begin{equation}
\label{cotton} \displaystyle{C_{ijk}=\nabla_{i}R_{jk}-\nabla_{j}R_{ik}-\frac{1}{2(n-1)}\big(\nabla_{i}R
g_{jk}-\nabla_{j}R g_{ik}).}
\end{equation} Using the Bianchi identity, one easily verifies that the Cotton
and Weyl tensors are related as follows:
\begin{equation}
\label{cottonwyel} \displaystyle{C_{ijk}=-\frac{(n-2)}{(n-3)}\nabla_{l}W_{ijkl},}
\end{equation} provided $n\ge 4.$ Moreover, it is easy to see that
\begin{equation}\label{simetC}
C_{ijk}=-C_{jik} \,\,\,\,\,\,\hbox{and}\,\,\,\,\,\,C_{ijk}+C_{jki}+C_{kij}=0.
\end{equation} In particular, we have
\begin{equation}\label{freetraceC}
g^{ij}C_{ijk}=g^{ik}C_{ijk}=g^{jk}C_{ijk}=0.
\end{equation} Next, the Schouten tensor $A$ is defined by
\begin{equation}
\label{schouten} A_{ij}=R_{ij}-\frac{R}{2(n-1)}g_{ij}.
\end{equation} Combining Eqs. (\ref{weyl}) and (\ref{schouten}) we have the following decomposition
\begin{equation}
\label{WS} R_{ijkl}=\frac{1}{n-2}(A\odot g)_{ijkl}+W_{ijkl},
\end{equation} where $\odot$ is the Kulkarni-Nomizu product.  Further, we recall that the Bach tensor on a Riemannian manifold $(M^n,g)$, $n\geq 4,$ is defined in term of the components of the Weyl tensor $W_{ikjl}$ as follows
\begin{equation}
\label{bach} B_{ij}=\frac{1}{n-3}\nabla_{k}\nabla_{l}W_{ikjl}+\frac{1}{n-2}R_{kl}W_{ikjl},
\end{equation}
while for $n=3$ it is given by
\begin{equation}
\label{bach3} B_{ij}=\nabla_{k}C_{kij}=div\,(C)_{ij}.
\end{equation} For more details about these tensors, we refer to \cite{bach,besse}. We say that $(M^n,g)$ is Bach-flat when $B_{ij}=0.$ It is easy to check that if  $(M^n,g)$ is either locally conformally flat or Einstein, then $(M^n,g)$ is Bach-flat. 

Next, we remember that a Riemannian manifold $(M^n,g)$ is a $V$-static metric if there exists a smooth function $f$ and a constant $\kappa$ such that
\begin{equation}
\label{defVstatic}
-(\Delta f)g+Hess\, f-f Ric=\kappa g.
\end{equation} In particular, tracing (\ref{defVstatic}) we deduce that the $V$-static potential $f$ also satisfies the linear equation
\begin{equation}
\label{eqtrace}
-(n-1)\Delta f=Rf+\kappa n.
\end{equation} From this, it is easy to check that

\begin{equation}
\label{eqVstaic2}Hess\,  f-fRic=-\frac{Rf+\kappa}{n-1}g
\end{equation}
and
\begin{equation}
\label{p1a}
f\mathring{Ric}=\mathring{Hess\,f},
\end{equation} where $\mathring{T}=T-\frac{{\rm tr}T}{n}g$ stands for the traceless of tensor T. 

\begin{remark} It is also important to recall that choosing appropriate coordinates, $f$ and $g$ are analytic, and hence, the set of regular points of $f$ is dense in $M^n.$ Thus, at regular points of $f,$ the vector field $\nu=-\frac{\nabla f}{|\nabla f|}$ is normal to $\partial M$ and $|\nabla f|$ is constant (non null) along $\partial M$. This is a consequence of an ODE argument by Corvino, Eichmair and Miao (see Proposition 2.1 in \cite{CEM}).
\end{remark}

\vspace{0,5cm}

Now we ready to proof Theorem \ref{thmA}.

\subsection{Proof of Theorem \ref{thmA}}

First of all, we shall present a crucial integral formula.

\begin{lemma}\label{L2}
Let $(M^{n},g,f)$ be a Miao-Tam critical metric. Then we have:
\begin{eqnarray*}
\int_{M}f^{2}|C|^{2}dV_{g}+\int_{M} f^{2}\div^{3}(C)dV_{g}+2\int_{M}\div\,C(\nabla f,\nabla f)dV_{g}=0.
\end{eqnarray*}
\end{lemma}
\begin{proof} Taking into account that $M^n$ has constant scalar curvature, we use (\ref{cotton}) together with (\ref{simetC}) to infer
\begin{eqnarray}
\label{mn1}
\int_{M} f^{2}|C|^{2}dV_{g}&=&\int_{M} f^{2}(\nabla_{i}R_{jk}-\nabla_{j}R_{ik})C_{ijk} dV_{g}\nonumber\\
&=&2\int_{M} f^{2}\nabla_{i} R_{jk}C_{ijk}dV_{g}.
\end{eqnarray} 

Next, by using the Stokes's formula and that $f$ vanishes on the boundary we derive
\begin{eqnarray*}
0&=&\int_{M}\nabla_{i}\big(f^{2}R_{jk}C_{ijk}\big)dV_g\nonumber\\ &=&\int_{M}\nabla_{i}(f^{2})R_{jk}C_{ijk}dV_g+\int_{M}f^{2}\nabla_{i}R_{jk}C_{ijk}dV_g \nonumber\\&&+\int_{M}f^{2}R_{jk}\nabla_{i}C_{ijk}dV_g,
\end{eqnarray*} and this substituted into (\ref{mn1}) gives

\begin{eqnarray}
\label{n1}
\int_{M} f^{2}|C|^{2}dV_{g}&=&-2\int_{M} \nabla_{i}(f^{2})R_{jk}C_{ijk}dV_{g}-2\int_{M} f^{2}R_{jk}\nabla_{i}C_{ijk}dV_{g}.
\end{eqnarray} Taking into account that the Cotton tensor is trace-free in any two indices, we use the fundamental equation (\ref{eqMiaoTam}) to obtain
\begin{equation}
\label{nb1}
\int_{M} f^{2}|C|^{2}dV_{g}=-4\int_{M}\nabla_{i}f\nabla_{j}\nabla_{k}fC_{ijk}dV_{g}-2\int_{M} f\nabla_{j}\nabla_{k}f\nabla_{i}C_{ijk}dV_{g}.
\end{equation}  

One notices that

\begin{equation}
\label{nb2}
\nabla_{j}\Big(\nabla_{i}f \nabla_{k}f C_{ijk}\Big)=\nabla_{j}\nabla_{i}f\nabla_{k}f C_{ijk}+\nabla_{i}f\nabla_{j}\nabla_{k}f C_{ijk}-\div\, C(\nabla f, \nabla f).
\end{equation} But, by using once more that the Cotton is skew-symmetric in the first two indices, it is easy to check, using the fact that the Hessian is symmetric, that $\nabla_{j}\nabla_{i}f\nabla_{k}f C_{ijk}=-\nabla_{j}\nabla_{i}f\nabla_{k}f C_{ijk},$ and hence, $\nabla_{j}\nabla_{i}f\nabla_{k}f C_{ijk}=0.$ Therefore, on integrating (\ref{nb2}) over $M^n,$ we may use this jointly with Stokes's formula to arrive at

\begin{eqnarray}
\label{nb3}
\int_{M}\nabla_{i}f\nabla_{j}\nabla_{k}f C_{ijk} dV_{g}&=& \int_{M}\div\, C(\nabla f, \nabla f) dV_{g}+\int_{\partial M}C(\nabla f, \nu,\nabla f)dS_g\nonumber\\&=&\int_{M}\div\, C(\nabla f, \nabla f) dV_{g},
\end{eqnarray} where $\nu=\mp\frac{\nabla f}{|\nabla f|},$ according to whether the potential is positive, respectively nega\-tive, on $M.$ Moreover, we used in the last line the anti-symmetry of the Cotton tensor.

In the meantime, notice that
\begin{eqnarray*}
\nabla_{j}\Big(f\nabla_{k}f \nabla_{i}C_{ijk}\Big)&=&\nabla_{j}f\nabla_{k}f\nabla_{i}C_{ijk}+f\nabla_{j}\nabla_{k}f\nabla_{i}C_{ijk}\nonumber\\&&+f\nabla_{k}f\nabla_{j}\nabla_{i}C_{ijk}\nonumber\\&=&\div\, C(\nabla f,\nabla f)+f\nabla_{j}\nabla_{k}f\nabla_{i}C_{ijk}\nonumber\\&&+\frac{1}{2}\nabla_{k}(f^{2})\nabla_{j}\nabla_{i}C_{ijk}.
\end{eqnarray*} Upon integrating the above expression over $M$ we use again the Stokes's formula and that $f$ vanishes on the boundary to deduce

\begin{eqnarray*}
\int_{M}f\nabla_{j}\nabla_{k}f\nabla_{i}C_{ijk} dV_g&=&-\int_{M}\div\, C(\nabla f,\nabla f) dV_{g}-\frac{1}{2}\int_{M}\nabla_{k}(f^{2})\nabla_{j}\nabla_{i}C_{ijk}dV_{g},
\end{eqnarray*} which can be rewritten as 

\begin{eqnarray}
\label{nb4}
\int_{M}f\nabla_{j}\nabla_{k}f\nabla_{i}C_{ijk}dV_g&=&-\int_{M}\div\, C(\nabla f,\nabla f)dV_{g}+\frac{1}{2}\int_{M}f^{2}\nabla_{k}\nabla_{j}\nabla_{i}C_{ijk}dV_{g}\nonumber\\&=&-\int_{M}\div\, C(\nabla f,\nabla f)dV_{g}+\frac{1}{2}\int_{M}f^{2}\div^{3}(C) dV_{g}.
\end{eqnarray}

Now, substituting (\ref{nb3}) and (\ref{nb4}) into (\ref{nb1}) we obtain

\begin{eqnarray*}
\int_{M} f^{2}|C|^{2}dV_{g}+2\int_{M} \div\, C(\nabla f,\nabla f)dV_{g}&=&\int_{M}\nabla_{k}(f^{2})\nabla_{j}\nabla_{i}C_{ijk}dV_{g},
\end{eqnarray*} that is,

\begin{eqnarray}
\label{nm5}
\int_{M} f^{2}|C|^{2}dV_{g}+2\int_{M} \div\, C(\nabla f,\nabla f)dV_{g}+\int_{M}f^{2}\div^{3}(C) dV_{g}=0.
\end{eqnarray} This is what we wanted to prove.

\end{proof}

Proceeding, since $M^4$ has second order divergence-free Weyl tensor, we may use (\ref{nm5}) jointly with (\ref{cottonwyel}) to deduce that the Cotton tensor $C$ is identically zero in $M^4.$ This implies that the Weyl tensor is harmonic, and then it suffices to apply Theorem 10.3 in \cite{Kim} to conclude that $M^4$ is isometric to a geodesic ball in a simply connected space form $\mathbb{R}^{4}$, $\mathbb{H}^{4}$ or $\mathbb{S}^{4}.$ This finishes the proof of Theorem \ref{thmA}.

\section{Integral Curvature Estimate}
\label{secThmB}

In this section, we will prove Theorem \ref{thmB}. To do so, we first need to recall some useful results obtained in \cite{BalRi2,BDR}.

\begin{lemma}[\cite{BalRi2,BDR}]
\label{lem1} Let $(M^{n},\,g,\,f,\kappa)$ be a connected, smooth Riemannian manifold and $f$ is a smooth function on $M^n$ satisfying the $V$-static equation (\ref{eqVstatic}). Then
\begin{eqnarray*}
f(\nabla_iR_{jk}-\nabla_jR_{ik})=R_{ijkl}\nabla_lf+\frac{R}{n-1}(\nabla_i f g_{jk}-\nabla_j f g_{ik})-(\nabla_iR_{jk}-\nabla_jR_{ik}).
\end{eqnarray*}
\end{lemma}

Proceeding, it is not difficult to check that Riemannian manifolds satisfying the $V$-static equation must satisfy
\begin{equation}
\label{eqT}
fC_{ijk}=T_{ijk}+W_{ijkl}\nabla_{l}f,
\end{equation} where the covariant 3-tensor $T_{ijk},$ defined previously in \cite{BDR}, is given by

\begin{eqnarray}
\label{T}
T_{ijk}&=&\frac{n-1}{n-2}\left(R_{ik}\nabla_{j}f-R_{jk}\nabla_{i}f\right)
-\frac{R}{n-2}\left(g_{ik}\nabla_{j}f-g_{jk}\nabla_{i}f\right)\nonumber\\
 &&+\frac{1}{n-2}\left(g_{ik}R_{js}\nabla^{s}f-g_{jk}R_{is}\nabla^{s}f\right).
\end{eqnarray}

Before presenting the proof of main results it is necessary to prove a couple of key lemmas. To do so, it is crucial to recall a B\"ochner type formula for $V$-static metrics obtained recently in \cite{BalRi2} (cf. Theorem 2 in \cite{BalRi2}). More precisely, Baltazar and Ribeiro proved that a connected, smooth Riemannian manifold $M^n$ and a smooth function $f$  on $M^n$ satisfying the $V$-static equation (\ref{eqVstatic}) must satisfy 

\begin{eqnarray}
\label{eqBalRib}
\frac{1}{2}\div (f\nabla|\ric|^2)&=&\frac{n-2}{n-1}f|C|^2+f|\nabla\ric|^2+\frac{n\kappa}{n-1}|\ric|^2\nonumber\\&&+\frac{2Rf}{n-1}|\ric|^2+\frac{2nf}{n-2}\rc_{ij}\rc_{jk}\rc_{ki}\nonumber\\&&-\frac{n-2}{n-1}W_{ijkl}C_{ijk}\nabla_lf-2fW_{ijkl}\rc_{ik}\rc_{jl}.
\end{eqnarray}
 
In the sequel, we shall use (\ref{eqBalRib}) to get the following lemma. 

\begin{lemma} \label{lemA1a} Let $(M^{n},\,g,\,f,\kappa)$ be a connected, smooth Riemannian manifold and $f$ is a smooth function on $M^n$ satisfying the $V$-static equation (\ref{eqVstatic}). Then we have:
\begin{eqnarray}\label{lemA1}
\frac{1}{2}\div(f^2\nabla|\ric|^2)&=&f^2|\nabla\ric|^2+\frac{n\kappa f}{n-1}|\ric|^2+\frac{2Rf^2}{n-1}|\ric|^2\nonumber\\&&+2f\nabla_i \rc_{jk}\rc_{ik}\nabla_j f-\frac{1}{2}\langle f\nabla|\ric|^2,\nabla f\rangle\nonumber\\&&-2f^2W_{ijkl}\rc_{ik}\rc_{jl}+\frac{2nf^2}{n-2}\rc_{ij}\rc_{jk}\rc_{ki}.
\end{eqnarray}
\end{lemma}
\begin{proof} To begin with, we use Eq. (\ref{eqBalRib}) to obtain
\begin{eqnarray}
\label{eq1}
\frac{1}{2}\div(f^2\nabla|\ric|^2)&=&\frac{1}{2}f \div(f\nabla|\ric|^2)+\frac{1}{2}\langle f\nabla|\ric|^2,\nabla f\rangle\nonumber\\
 &=&f\left[\frac{n-2}{n-1}f|C|^2+f|\nabla\ric|^2+\frac{n\kappa}{n-1}|\ric|^2+\frac{2Rf}{n-1}|\ric|^2\right.\nonumber\\
 &&\left.+\frac{2nf}{n-2}\rc_{ij}\rc_{jk}\rc_{ki}-\frac{n-2}{n-1}W_{ijkl}C_{ijk}\nabla_lf-2fW_{ijkl}\rc_{ik}\rc_{jl}\right]\nonumber\\
 &&+\frac{1}{2}\langle f\nabla|\ric|^2,\nabla f\rangle\nonumber\\&=&\frac{n-2}{n-1}f^2|C|^2+f^2|\nabla\ric|^2+\frac{n\kappa f}{n-1}|\ric|^2+\frac{2Rf^2}{n-1}|\ric|^2\nonumber\\
 &&+\frac{2nf^2}{n-2}\rc_{ij}\rc_{jk}\rc_{ki}-\frac{n-2}{n-1}fW_{ijkl}C_{ijk}\nabla_lf\\
 &&-2f^2W_{ijkl}\rc_{ik}\rc_{jl}+\frac{1}{2}\langle f\nabla|\ric|^2,\nabla f\rangle\nonumber.
\end{eqnarray}

On the other hand, by using Eq. (\ref{eqT}) we have $$W_{ijkl}C_{ijk}\nabla_l f=(fC_{ijk}-T_{ijk})C_{ijk}
 =f|C|^2-T_{ijk}C_{ijk}.$$ Next, since $M^n$ has constant scalar curvature, from the Bianchi identity together with (\ref{cotton}) and (\ref{T}) we infer

\begin{eqnarray*}
 W_{ijkl}C_{ijk}\nabla_lf&=&f|C|^2-\frac{n-1}{n-2}(R_{ik}\nabla_jf-R_{jk}\nabla_if)(\nabla_iR_{jk}-\nabla_jR_{ik})\\
 &=&f|C|^2-\frac{2(n-1)}{n-2}(\nabla_iR_{jk}R_{ik}\nabla_jf-\nabla_jR_{ik}R_{ik}\nabla_jf)\\
 &=&f|C|^2-\frac{2(n-1)}{n-2}\nabla_iR_{jk}R_{ik}\nabla_jf+\frac{2(n-1)}{n-2}\nabla_jR_{ik}R_{ik}\nabla_jf.
\end{eqnarray*} From this it follows that
\begin{equation}\label{eq2}
\frac{n-2}{n-1}fW_{ijkl}C_{ijk}\nabla_l f =\frac{n-2}{n-1}f^2|C|^2-2f\nabla_iR_{jk}R_{ik}\nabla_j f+2f\nabla_jR_{ik}R_{ik}\nabla_jf.
\end{equation} Substituting (\ref{eq2}) into (\ref{eq1}) we get
\begin{eqnarray}\label{eq3}
\frac{1}{2}\div(f^2\nabla|\ric|^2)&=&\frac{n-2}{n-1}f^2|C|^2+f^2|\nabla\ric|^2+\frac{n\kappa f}{n-1}|\ric|^2+\frac{2Rf^2}{n-1}|\ric|^2\nonumber\\
 &&+\frac{2nf^2}{n-2}\rc_{ij}\rc_{jk}\rc_{ki}-\frac{n-2}{n-1}f^2|C|^2+2f\nabla_iR_{jk}R_{ik}\nabla_jf\nonumber\\
 &&-2f\nabla_jR_{ik}R_{ik}\nabla_jf-2f^2W_{ijkl}\rc_{ik}\rc_{jl}+\frac{1}{2}\langle f\nabla|\ric|^2,\nabla f\rangle\nonumber\\
 &=&f^2|\nabla\ric|^2+\frac{n\kappa f}{n-1}|\ric|^2+\frac{2Rf^2}{n-1}|\ric|^2\nonumber\\&&+2f\nabla_i \rc_{jk}\rc_{ik}\nabla_j f-\frac{1}{2}\langle f\nabla|\ric|^2,\nabla f\rangle\nonumber\\&&-2f^2W_{ijkl}\rc_{ik}\rc_{jl}+\frac{2nf^2}{n-2}\rc_{ij}\rc_{jk}\rc_{ki},
\end{eqnarray} where we used that $M^n$ has constant scalar curvature to deduce $\nabla Ric=\nabla \mathring{Ric}$ and the Bianchi identity to infer $\nabla_{j}R_{ik}\cdot g_{ik}=0.$ So, the proof is finished. 
\end{proof}

Proceeding, we shall deduce an integral formula for Miao-Tam critical metrics, which plays a fundamental role in the proof of Theorem \ref{thmB}.

\begin{lemma} \label{lemB1} 
Let $(M^{n},\,g,\,f)$ be a compact, oriented, connected Miao-Tam critical metric with smooth boundary $\partial M.$ Then we have:
\begin{eqnarray}\label{lemA1}
2\int_Mf^2W_{ijkl}\rc_{ik}\rc_{jl}dV_g&=&-2\int_M|\ric(\nabla f)|^2dV_g + \int_M f^2|\nabla\ric|^2dV_g\nonumber\\
 &&+\frac{n}{n-1}\int_Mf|\ric|^2dV_g+\frac{2R}{n-1}\int_Mf^2|\ric|^2dV_g\nonumber\\
 &&+\frac{n-4}{2n}\int_Mf\Delta f|\ric|^2dV_g+\frac{1}{2}\int_M|\ric|^2|\nabla f|^2 dV_g\nonumber\\
 &&+\frac{4}{n-2}\int_Mf^2\rc_{ij}\rc_{jk}\rc_{ki}dV_g.
\end{eqnarray}
\end{lemma}

\begin{proof} We start integrating the expression obtained in Lemma \ref{lemA1a} over $M$ and using that $f$ vanishes on the boundary $\partial M$ to achieve
\begin{eqnarray}\label{eq4}
0&=&\int_Mf^2|\nabla\ric|^2dV_g+\frac{n}{n-1}\int_Mf|\ric|^2dV_g+\frac{2R}{n-1}\int_Mf^2|\ric|^2dV_g\nonumber\\
 &&+2\int_Mf\nabla_i\rc_{jk}\rc_{ik}\nabla_jfdV_g-\frac{1}{2}\int_M\langle f\nabla|\ric|^2,\nabla f\rangle dV_g\nonumber\\
 &&-2\int_Mf^2W_{ijkl}\rc_{ik}\rc_{jl}dV_g+\frac{2n}{n-2}\int_Mf^2\rc_{ij}\rc_{jk}\rc_{ki}dV_g.
\end{eqnarray}

Easily one verifies that
\begin{eqnarray*}
\int_Mf\nabla_i\rc_{jk}\rc_{ik}\nabla_jfdV_g&=&\int_M\nabla_i(f\rc_{jk}\rc_{ik}\nabla_jf)dV_g-\int_M\rc_{jk}\rc_{ik}\nabla_if\nabla_jfdV_g\nonumber\\
 &&-\int_Mf\rc_{jk}\rc_{ik}\nabla_i\nabla_jfdV_g,
\end{eqnarray*} and this combined with (\ref{eqVstaic2}) yields
\begin{eqnarray*}
\int_Mf\nabla_i\rc_{jk}\rc_{ik}\nabla_jfdV_g &=&-\int_M|\ric(\nabla f)|^2dV_g-\int_Mf\rc_{jk}\rc_{ik}(f\rc_{ij}+\frac{\Delta f}{n}g_{ij})dV_g,
\end{eqnarray*} which can be rewritten as

\begin{equation}
\label{eq5}
\int_Mf\nabla_i\rc_{jk}\rc_{ik}\nabla_jfdV_g = -\int_M|\ric(\nabla f)|^2dV_g-\int_Mf^2\rc_{jk}\rc_{ik}\rc_{ij}dV_g -\frac{1}{n}\int_Mf\Delta f|\ric|^2dV_g.
\end{equation} Putting this into (\ref{eq4}) we conclude
\begin{eqnarray*}
0&=&\int_Mf^2|\nabla\ric|^2dV_g+\frac{n}{n-1}\int_Mf|\ric|^2dV_g+\frac{2R}{n-1}\int_Mf^2|\ric|^2dV_g\nonumber\\
 &&-2\int_M|\ric(\nabla f)|^2dV_g-2\int_Mf^2\rc_{jk}\rc_{ik}\rc_{ij}dV_g-\frac{2}{n}\int_Mf\Delta f|\ric|^2dV_g\nonumber\\
 &&-\frac{1}{2}\int_M\langle f\nabla|\ric|^2,\nabla f\rangle dV_g-2\int_Mf^2W_{ijkl}\rc_{ik}\rc_{jl}dV_g\nonumber\\
 &&+\frac{2n}{n-2}\int_Mf^2\rc_{ij}\rc_{jk}\rc_{ki}dV_g,
\end{eqnarray*} in other words,

\begin{eqnarray}\label{eq6}
 0&=&\int_Mf^2|\nabla\ric|^2dV_g+\frac{n}{n-1}\int_Mf|\ric|^2dV_g+\frac{2R}{n-1}\int_Mf^2|\ric|^2dV_g\nonumber\\
 &&-\frac{1}{2}\int_M\langle f\nabla|\ric|^2,\nabla f\rangle dV_g-2\int_M|\ric(\nabla f)|^2dV_g-\frac{2}{n}\int_Mf\Delta f|\ric|^2dV_g\nonumber\\
 &&-2\int_Mf^2W_{ijkl}\rc_{ik}\rc_{jl}dV_g+\frac{4}{n-2}\int_Mf^2\rc_{ij}\rc_{jk}\rc_{ki}dV_g.
\end{eqnarray}

To proceed, by a direct computation, we can check that

\begin{eqnarray*}
\div(f|\ric|^2\nabla f)&=&f\Delta f|\ric|^2+f\langle\nabla|\ric|^2,\nabla f\rangle+|\ric|^2|\nabla f|^2,
\end{eqnarray*} from which we deduce
\begin{eqnarray*}
0&=&\int_Mf\Delta f|\ric|^2dV_g+\int_Mf\langle\nabla|\ric|^2,\nabla f\rangle dV_g+\int_M|\ric|^2|\nabla f|^2dV_g.
\end{eqnarray*} Thus,
\begin{equation}
\label{eq7}
\int_M f\langle\nabla|\ric|^2,\nabla f\rangle dV_g=-\int_Mf\Delta f|\ric|^2dV_g-\int_M|\ric|^2|\nabla f|^2dV_g.
\end{equation} One notices that combining (\ref{eq7}) with (\ref{eq6}) we arrive at
\begin{eqnarray*}
\label{eq8}
0&=&\int_Mf^2|\nabla\ric|^2dV_g+\frac{n}{n-1}\int_Mf|\ric|^2dV_g+\frac{2R}{n-1}\int_Mf^2|\ric|^2 dV_g\nonumber\\
 &&+\frac{1}{2}\int_Mf\Delta f|\ric|^2 dV_g+\frac{1}{2}\int_M|\ric|^2|\nabla f|^2 dV_g -2\int_M|\ric(\nabla f)|^2dV_g\nonumber\\
 &&-\frac{2}{n}\int_Mf\Delta f|\ric|^2 dV_g-2\int_Mf^2W_{ijkl}\rc_{ik}\rc_{jl}dV_g\nonumber\\
 &&+\frac{4}{n-2}\int_Mf^2\rc_{ij}\rc_{jk}\rc_{ki}dV_g\nonumber\\
 &=&\int_M f^2|\nabla\ric|^2 dV_g+\frac{n}{n-1}\int_Mf|\ric|^2 dV_g+\frac{2R}{n-1}\int_Mf^2|\ric|^2 dV_g\nonumber\\
 &&+\frac{n-4}{2n}\int_Mf\Delta f|\ric|^2 dV_g+\frac{1}{2}\int_M|\ric|^2|\nabla f|^2 dV_g-2\int_M|\ric(\nabla f)|^2dV_g\nonumber\\
 &&-2\int_Mf^2W_{ijkl}\rc_{ik}\rc_{jl}dV_g+\frac{4}{n-2}\int_Mf^2\rc_{ij}\rc_{jk}\rc_{ki}dV_g,\nonumber
\end{eqnarray*} that was to be proved.

\end{proof}

\subsection{Proof of Theorem \ref{thmB}}
\begin{proof} Initially, we assume that $M^n$ is a compact Miao-Tam cri\-tical metric with positive scalar curvature and dimension $n\ge 4.$ Thus, we may use Lemma \ref{lemB1} jointly with the classical Kato inequality $|\nabla |\ric||\leq |\nabla \ric|$ to obtain
\begin{eqnarray}\label{eq9}
0&\geq&\int_Mf^2|\nabla|\ric||^2 dV_g+\frac{n}{n-1}\int_Mf|\ric|^2dV_g\nonumber\\
 &&+\frac{2R}{n-1}\int_Mf^2|\ric|^2 dV_g+\frac{n-4}{2n}\int_Mf\Delta f|\ric|^2 dV_g\nonumber\\
 &&+\frac{1}{2}\int_M|\ric|^2|\nabla f|^2 dV_g-2\int_M|\ric(\nabla f)|^2dV_g\nonumber\\
 &&-2\int_Mf^2W_{ijkl}\rc_{ik}\rc_{jl}dV_g+\frac{4}{n-2}\int_Mf^2\rc_{ij}\rc_{jk}\rc_{ki}dV_g.
\end{eqnarray}

On the other hand, by a result by Catino (cf. \cite{catinoAdv}, Proposition 2.1, see also \cite{bour}, Lemma 4.7), on every Riemannian manifold of dimension $n\geq 4$ the following estimate holds 

\begin{eqnarray*}
\Big|-W_{ijkl}\rc_{ik}\rc_{jl}+\frac{2}{n-2}\rc_{ij}\rc_{jk}\rc_{ik}\Big|\le \sqrt{\frac{n-2}{2(n-1)}}\Big(|W|^{2}+\frac{8}{n(n-2)}|\ric|^{2}\Big)^{\frac{1}{2}}|\ric|^{2}.
\end{eqnarray*} This employed into (\ref{eq9}) achieves

\begin{eqnarray}\label{eq10}
0&\geq&\int_Mf^2|\nabla|\ric||^2dV_g+\frac{n}{n-1}\int_Mf|\ric|^2dV_g+\frac{2R}{n-1}\int_Mf^2|\ric|^2dV_g\nonumber\\
 &&+\frac{n-4}{2n}\int_Mf\Delta f|\ric|^2 dV_g+\frac{1}{2}\int_M|\ric|^2|\nabla f|^2 dV_g-2\int_M|\ric(\nabla f)|^2dV_g\nonumber\\
 &&-\sqrt{\frac{2(n-2)}{n-1}}\int_Mf^2\left(|W|^2+\frac{8}{n(n-2)}|\ric|^2\right)^\frac{1}{2}|\ric|^2dV_g.
\end{eqnarray}

For what follows it is essential to remark that from (\ref{Yamabeconst}) we have 
\begin{eqnarray}
\label{desYamabe}
\frac{n-2}{4(n-1)}\mathcal{Y}(M,\partial M,[g])\left(\int_M|\phi|^\frac{2n}{n-2}dV_g\right)^\frac{n-2}{n}&\leq&\int_M|\nabla\phi|^2dV_g+\frac{n-2}{4(n-1)}\int_M R\phi^2dV_g\nonumber\\
 &&+\frac{n-2}{2(n-1)}\int_{\partial M}H\phi^2 dS_g,
\end{eqnarray} where $H$ is the mean curvature of $\partial M.$ Since $f^{-1}(0)=\partial M,$ we deduce that $f$ does not change sign. Then, we assume that $f$ is nonnegative. In particular, $f > 0$ at the interior of $M.$ Hence, by choosing $\phi=f|\ric|$ in (\ref{desYamabe}) and therefore, using that $f$ vanishes on the boundary and that $M$ has constant scalar curvature, we obtain
\begin{eqnarray*}
\int_M|\nabla(f|\ric|)|^2 dV_g&\geq&\frac{n-2}{4(n-1)}\mathcal{Y}(M,\partial M,[g])\left(\int_M f^\frac{2n}{n-2}|\ric|^\frac{2n}{n-2}dV_g\right)^\frac{n-2}{n}\\
 &&-\frac{(n-2)R}{4(n-1)}\int_M f^2|\ric|^2dV_g.
\end{eqnarray*} Moreover, taking into account that

\begin{eqnarray*}
|\nabla(f|\ric|)|^2&=&|f\nabla|\ric|+|\ric|\nabla f|^2\\
 &=&f^2|\nabla|\ric||^2+2f|\ric|\langle\nabla|\ric|,\nabla f\rangle+|\ric|^2|\nabla f|^2\\
 &=&f^2|\nabla|\ric||^2+f\langle\nabla|\ric|^2,\nabla f\rangle+|\ric|^2|\nabla f|^2,
\end{eqnarray*} we arrive at
\begin{eqnarray*}
\int_Mf^2|\nabla|\ric||^2 dV_g &\geq&\frac{n-2}{4(n-1)}\mathcal{Y}(M,\partial M,[g])\left(\int_Mf^\frac{2n}{n-2}|\ric|^\frac{2n}{n-2} dV_g\right)^\frac{n-2}{n}\nonumber\\
 &&-\frac{(n-2)R}{4(n-1)}\int_Mf^2|\ric|^2 dV_g-\int_M f\langle\nabla|\ric|^2,\nabla f\rangle dV_g\nonumber\\
 &&-\int_M|\ric|^2|\nabla f|^2 dV_g.
\end{eqnarray*} By using (\ref{eq7}) we have
\begin{eqnarray}\label{eq12}
\int_M f^2|\nabla|\ric||^2 dV_g &\geq&\frac{n-2}{4(n-1)}\mathcal{Y}(M,\partial M,[g])\left(\int_Mf^\frac{2n}{n-2}|\ric|^\frac{2n}{n-2}dV_g\right)^\frac{n-2}{n}\nonumber\\
 &&-\frac{(n-2)R}{4(n-1)}\int_Mf^2|\ric|^2dV_g+\int_Mf\Delta f|\ric|^2dV_g.
\end{eqnarray} Next, combining (\ref{eq10}) and (\ref{eq12}) we immediately obtain
\begin{eqnarray*}
0&\geq&\frac{n-2}{4(n-1)}\mathcal{Y}(M,\partial M,[g])\left(\int_Mf^\frac{2n}{n-2}|\ric|^\frac{2n}{n-2}dV_g\right)^\frac{n-2}{n}\\
 &&-\frac{(n-2)R}{4(n-1)}\int_Mf^2|\ric|^2dV_g+\int_Mf\Delta f|\ric|^2dV_g+\frac{n}{n-1}\int_Mf|\ric|^2dV_g\\
 &&+\frac{2R}{n-1}\int_Mf^2|\ric|^2dV_g+\frac{n-4}{2n}\int_Mf\Delta f|\ric|^2dV_g\\
 &&+\frac{1}{2}\int_M|\ric|^2|\nabla f|^2dV_g-2\int_M|\ric(\nabla f)|^2dV_g\\
 &&-\sqrt{\frac{2(n-2)}{n-1}}\int_Mf^2\left(|W|^2+\frac{8}{n(n-2)}|\ric|^2\right)^\frac{1}{2}|\ric|^2dV_g,
\end{eqnarray*} which can be rewritten as 

\begin{eqnarray*}
0&\geq&\frac{n-2}{4(n-1)}\mathcal{Y}(M,\partial M,[g])\left(\int_Mf^\frac{2n}{n-2}|\ric|^\frac{2n}{n-2}dV_g\right)^\frac{n-2}{n}+\frac{1}{2}\int_M|\ric|^2|\nabla f|^2dV_g\\
 &&-2\int_M|\ric(\nabla f)|^2dV_g+\frac{n}{n-1}\int_Mf|\ric|^2dV_g-\frac{(n-10)R}{4(n-1)}\int_Mf^2|\ric|^2dV_g\\
 &&-\sqrt{\frac{2(n-2)}{n-1}}\int_Mf^2\left(|W|^2+\frac{8}{n(n-2)}|\ric|^2\right)^\frac{1}{2}|\ric|^2dV_g\\
 &&+\frac{3n-4}{2n}\int_Mf\Delta f|\ric|^2dV_g,
\end{eqnarray*} and by Eq. (\ref{eqtrace}) for $\kappa=1,$ we get

\begin{eqnarray*}
0&\geq&\frac{n-2}{4(n-1)}\mathcal{Y}(M,\partial M,[g])\left(\int_Mf^\frac{2n}{n-2}|\ric|^\frac{2n}{n-2}dV_g\right)^\frac{n-2}{n}+\frac{1}{2}\int_M|\ric|^2|\nabla f|^2dV_g\nonumber\\
 &&-2\int_M|\ric(\nabla f)|^2dV_g+\frac{n}{n-1}\int_Mf|\ric|^2dV_g-\frac{(n-10)R}{4(n-1)}\int_Mf^2|\ric|^2dV_g\nonumber\\
 &&-\frac{(3n-4)R}{2n(n-1)}\int_Mf^2|\ric|^2dV_g-\frac{(3n-4)}{2(n-1)}\int_Mf|\ric|^2dV_g\nonumber\\
 &&-\sqrt{\frac{2(n-2)}{n-1}}\int_Mf^2\left(|W|^2+\frac{8}{n(n-2)}|\ric|^2\right)^\frac{1}{2}|\ric|^2dV_g.
\end{eqnarray*} From here we deduce

\begin{eqnarray}
\label{eq13}
0&\geq &\frac{n-2}{4(n-1)}\mathcal{Y}(M,\partial M,[g])\left(\int_Mf^\frac{2n}{n-2}|\ric|^\frac{2n}{n-2}dV_g\right)^\frac{n-2}{n}+\frac{1}{2}\int_M|\ric|^2|\nabla f|^2dV_g\nonumber\\
 &&-2\int_M|\ric(\nabla f)|^2dV_g-\frac{(n-4)}{2(n-1)}\int_Mf|\ric|^2dV_g\nonumber\\
 &&-\frac{(n^2-4n-8)R}{4n(n-1)}\int_Mf^2|\ric|^2dV_g\nonumber\\
 &&-\sqrt{\frac{2(n-2)}{n-1}}\int_Mf^2\left(|W|^2+\frac{8}{n(n-2)}|\ric|^2\right)^\frac{1}{2}|\ric|^2dV_g.
 \end{eqnarray} Next, the H\"older inequality implies that

\begin{eqnarray}
\label{eq14m}
0&\geq &\frac{n-2}{4(n-1)}\mathcal{Y}(M,\partial M,[g])\left(\int_Mf^\frac{2n}{n-2}|\ric|^\frac{2n}{n-2}dV_g\right)^\frac{n-2}{n}\nonumber\\ &-&\sqrt{\frac{2(n-2)}{n-1}}\left(\int_M\left(|W|^2+\frac{8}{n(n-2)}|\ric|^2\right)^\frac{n}{4}dV_g\right)^\frac{2}{n}\left(\int_Mf^\frac{2n}{n-2}|\ric|^\frac{2n}{n-2}dV_g\right)^\frac{n-2}{n}\nonumber\\
 &+&\frac{1}{2}\int_M|\ric|^2|\nabla f|^2dV_g-2\int_M|\ric(\nabla f)|^2dV_g-\frac{(n-4)}{2(n-1)}\int_Mf|\ric|^2dV_g \nonumber\\
 &-&\frac{(n^2-4n-8)R}{4n(n-1)}\int_M f^2|\ric|^2 dV_g.\nonumber\\
\end{eqnarray}

Now, we claim that
\begin{equation}\label{eq15}
|\ric(\nabla f)|^2\leq\frac{(n-1)\sqrt{2n}}{2n}|\ric|^2|\nabla f|^2.
\end{equation} To prove this claim, first, notice that a straightforward computation gives
 $$|\ric(\nabla f)|^2=-\frac{1}{8}((df\otimes df)\odot g)_{ijkl}(\ric\odot\ric)_{ijkl}$$ and $$|((df\otimes df)\odot g)_{ijkl}(\ric\odot\ric)_{ijkl}|^2\leq|(df\otimes df)\odot g|^2|\ric\odot\ric|^2,$$ where $\odot$ is the Kulkarni-Nomizu product. We further have $$|(df\otimes df)\odot g|^2=4(n-1)|\nabla f|^4$$ and $$|\ric\odot\ric|^2=8|\ric|^4-8|\ric^2|^2,$$ where $(\ric^2)_{ik}=(\ric)_{ip}(\ric)_{kp}.$ In particular, it is easy to check that $tr \ric^2=|\ric |^{2}$ and then we immediately have $|\ric^2|^{2}\geq \frac{|\ric|^{4}}{n},$ which allows to deduce that 
 $$|\ric\odot\ric|^2\leq \frac{8(n-1)}{n}|\ric|^4.$$ Whence, it follows that $$|((df\otimes df)\odot g)_{ijkl}(\ric\odot\ric)_{ijkl}|^2\leq \frac{32(n-1)^2}{n}|\ric|^4|\nabla f|^4.$$ Thereby, we immediately achieve
\begin{equation*}
|\ric(\nabla f)|^2\leq\frac{(n-1)\sqrt{2n}}{2n}|\ric|^2|\nabla f|^2,
\end{equation*} which settles our claim.

Substituting (\ref{eq15}) into (\ref{eq14m}) we get

\begin{eqnarray}
\label{eq12f}
 0&\geq &\frac{n-2}{4(n-1)}\mathcal{Y}(M,\partial M,[g])\left(\int_Mf^\frac{2n}{n-2}|\ric|^\frac{2n}{n-2}dV_g\right)^\frac{n-2}{n}\nonumber\\
  &&-\sqrt{\frac{2(n-2)}{n-1}}\left(\int_M\left(|W|^2+\frac{8}{n(n-2)}|\ric|^2\right)^\frac{n}{4}dV_g\right)^\frac{2}{n}\left(\int_Mf^\frac{2n}{n-2}|\ric|^\frac{2n}{n-2}dV_g\right)^\frac{n-2}{n}\nonumber\\&&-\frac{2(n-1)\sqrt{2n}-n}{2n}\int_M|\ric|^2|\nabla f|^2dV_g-\frac{(n-4)}{2(n-1)}\int_Mf|\ric|^2dV_g\nonumber\\
 &&-\frac{(n^2-4n-8)R}{4n(n-1)}\int_M f^2|\ric|^2 dV_g.\nonumber\\
\end{eqnarray} In particular, by choosing $n=4$ in the above expression and using that $f$ is nonnegative, we achieve
\begin{eqnarray}\label{desn4}
0&\geq&\left[\frac{1}{6}\mathcal{Y}(M,\partial M,[g])-\frac{2}{\sqrt{3}}\left(\int_M\left(|W|^2+|\ric|^2\right)dV_g\right)^\frac{1}{2}\right]\left(\int_Mf^4|\ric|^4dV_g\right)^\frac{1}{2}\nonumber\\
 &&-\frac{3\sqrt{2}-1}{2}\int_M|\ric|^2|\nabla f|^2dV_g+\frac{R}{6}\int_Mf^2|\ric|^2 dV_g,\nonumber\\
\end{eqnarray} which, since $R>0,$ gives the desired inequality (\ref{thmeq}). 

Proceeding, we suppose that (\ref{thmeq}) is actually an equality. Then from (\ref{desn4}), using that $R>0,$ we see that $$\int_Mf^2|\ric|^2dV_g=0.$$ Whence, we have $|\ric|^2=0$ and then $M^4$ is Einstein. Finally, we may invoke Theorem 1.1 in \cite{miaotamTAMS} to conclude that $M^4$ is isometric to a geodesic ball in $\mathbb{S}^4.$ 

So, the proof is completed.
\end{proof}

\begin{acknowledgement}
The authors want to thank the referee for careful rea\-ding, rele\-vant remarks and valuable suggestions. Moreover, they want to thank A. Barros, E. Barbosa and R. Batista for fruitful conversations about this subject.
\end{acknowledgement}

\end{document}